\documentclass{amsart}

\usepackage{amssymb}
\usepackage{amsmath}
\usepackage{url}
\usepackage{enumerate}
\usepackage{graphicx}
\usepackage[utf8]{inputenc}
\usepackage[noend]{algpseudocode}
\usepackage{multicol}
\usepackage{datetime}
\usepackage{float}
\usepackage[normalem]{ulem}
\usepackage{multicol}
\usepackage{xcolor}
\usepackage{mathtools}

\usepackage[dvipdfmx]{hyperref}
\hypersetup{
    colorlinks  =   true,
    linkcolor   =   blue,
    filecolor   =   blue,
    urlcolor    =   blue,
    citecolor   =   blue,
}

\usepackage{newfloat}
\DeclareFloatingEnvironment[
  name = Procedure
]{procedure}

\numberwithin{procedure}{section}

\newcommand{\erase}[1]{}

\newtheorem{theorem}{Theorem}[section]

\newtheorem{proposition}[theorem]{Proposition}
\newtheorem{corollary}[theorem]{Corollary}

\newtheorem{_algorithm}[theorem]{Algorithm}

\newtheorem{_definition}[theorem]{Definition}
\newenvironment{definition}{\begin{_definition}\rm}{\end{_definition}}

\newtheorem{_propositiondefinition}[theorem]{Proposition-Definition}

\newtheorem{_remark}[theorem]{\it Remark}
\newenvironment{remark}{\begin{_remark}\rm}{\end{_remark}}

\newtheorem{_example}[theorem]{Example}

\newtheorem{_assumption}[theorem]{Assumption}

\newtheorem{_construction}[theorem]{Construction}

\newtheorem{_claim}[theorem]{Claim}

\newtheorem{_conjecture}[theorem]{Conjecture}

\numberwithin{equation}{section}
\numberwithin{table}{section}
\numberwithin{figure}{section}
\renewcommand{\qed}{\hfill {$\Box$}}

\newcommand{\F}{\mathord{\mathbb F}}

\newcommand{\Q}{\mathord{\mathbb  Q}}

\newcommand{\Z}{\mathord{\mathbb Z}}

\newcommand{\AAA}{\mathord{\mathcal A}}

\newcommand{\DDD}{\mathord{\mathcal D}}

\newcommand{\VVV}{\mathord{\mathcal V}}
\newcommand{\WWW}{\mathord{\mathcal W}}

\newcommand{\SSSS}{\mathord{\mathfrak S}}

\newcommand{\mapdownsurj}{
\hbox{$\bigm\downarrow$}
\llap{\hbox{\raise 2pt\hbox{$\bigm\downarrow$}}}%
\vstrechmapdown
}

\newcommand{\mapupsurj}{
\hbox{$\bigm\uparrow$}
\llap{\hbox{\raise 2pt\hbox{$\bigm\uparrow$}}}%
\vstrechmapup
}

\newcommand{\inj}{\hookrightarrow}

\newcommand{\set}[2]{\{\,{#1}\mid {#2} \,\}}

\newcommand{\gen}[1]{\langle {#1}  \rangle}

\newcommand{\tensor}{\otimes}

\newcommand{\sprime}{\sp\prime}

\newcommand{\spar}[1]{\sp{(#1)}}

\newcommand{\spprime}{\sp{\prime\prime}}

\newcommand{\sperp}{\sp{\perp}}

\newcommand{\semidirectproduct}{\rtimes}

\newcommand{\inv}{\sp{-1}}

\newcommand{\Hom}{\mathord{\mathrm{Hom}}}

\newcommand{\OG}{\mathord{\mathrm{O}}}

\newcommand{\Ker}{\operatorname{\mathrm{Ker}}\nolimits}

\newcommand{\Stab}{\operatorname{\rm Stab}}

\newcommand{\intfE}[1]{\langle #1\rangle_{E}}
\newcommand{\intfQ}[1]{\langle #1\rangle_{Q}}
\newcommand{\intfL}[1]{\langle #1\rangle_{L}}
\newcommand{\intfU}[1]{\langle #1\rangle_{U}}
\newcommand{\intfS}[1]{\langle #1\rangle_{S}}
\newcommand{\intfUT}[1]{\langle #1\rangle_{U_T}}

\newcommand{\baru}{\bar{u}}
\newcommand{\barx}{\bar{x}}
\newcommand{\bary}{\bar{y}}
\newcommand{\barz}{\bar{z}}
\newcommand{\barM}{\overline{M}}

\newcommand{\ptwo}{\mathrm{p}_2}
\newcommand{\psix}{\mathrm{p}_6}

\newcommand{\typeI}{\mathrm{I}}
\newcommand{\typeII}{\mathrm{II}}
\newcommand{\Min}{\operatorname{\mathrm{Min}}}

\newcommand{\tthh}{^{\mathrm{th}}}
\newcommand{\sstt}{^{\mathrm{st}}}
\newcommand{\onethird}{\frac{\,1\,}{3}}

\usepackage{mathtools}

\DeclarePairedDelimiter\floor{\lfloor}{\rfloor}

\begin{document}

\title[Quebbemann's extremal lattices]
{A note on Quebbemann's extremal \\ lattices of rank $64$}

\author[I. Shimada]{Ichiro Shimada}
\address{Department of Mathematics, Graduate School of Science, Hiroshima University,
1-3-1 Kagamiyama, Higashi-Hiroshima, 739-8526 JAPAN}
\email{ichiro-shimada@hiroshima-u.ac.jp}

\begin{abstract}
By constructing explicit examples,
we show that the  method of Quebbemann yields
many isomorphism classes  of 
extremal lattices of rank $64$.
Many of these examples  have no non-trivial automorphisms.
\end{abstract}

\thanks{This work was supported by JSPS KAKENHI Grant Number 16H03926 and~20H01798.}



\maketitle
%
%
\section{Introduction}\label{sec:Introduction}
By a \emph{lattice}, we mean an integral positive-definite lattice.
Let $L$ be an even unimodular lattice with the  bilinear form $\intfL{\, ,\,}\colon L\times L\to \Z$.
We put
\[
\min(L):=\min \set{\intfL{x, x}}{x\in L,\; x\ne 0}.
\]
It is well-known that the rank $n$ of $L$ is divisible by $8$, and that
$\min(L)$ satisfies
\begin{equation}\label{eq:min}
\min (L) \le 2+2 \floor*{\frac{n}{24}}.
\end{equation}
We say that an even unimodular lattice $L$ of rank $n$ is \emph{extremal}
if the equality holds in~\eqref{eq:min}.
Extremal lattices are an important research subject,
because they give rise to sphere packings of high density.
\par
Not so many \emph{explicit} examples of extremal lattices are known.
Moreover, since the construction of these examples involves very special algebraic objects, 
each of the known examples has a large automorphism group.  
For example, 
the automorphism group $\OG(\Lambda)$ of the Leech lattice $\Lambda$ 
is of order $2^{22} \cdot 3^9 \cdot 5^4 \cdot 7^2 \cdot 11\cdot 13 \cdot 23$.
\par
On the other hand, 
it was shown in ~\cite{Peters1983} 
that the number of isomorphism classes of extremal lattices of rank $40$ is $>8.45 \times 10^{51}$.
Since this bound was proved by means of a mass formula,
we do not obtain 
explicit examples  of extremal lattices of rank $40$
from this result.
\par
We consider extremal lattices of rank $64$.
Quebbemann~\cite{Q84} gave a method 
to construct extremal lattices $Q$ of rank $64$ 
from certain ternary codes $B$.
We call an extremal lattice of rank $64$ a \emph{Quebbemann lattice}
if it is obtained by (a generalization of) this method.
See Section~\ref{sec:Q} for the precise definition.
A remarkable property of Quebbemann's construction is 
that the condition on the ternary code $B$ required 
in order for the lattice $Q$ to be extremal is an \emph{open} condition.
Therefore we expect that, 
by generating sufficiently general ternary codes $B$,
we obtain many extremal lattices of rank $64$.
Another nice feature of  this construction is that 
we can calculate the set $\Min(Q)$ of non-zero minimal-norm vectors of a Quebbemann lattice $Q$ 
(that is, the set of vectors $v$ with $\intfQ{v, v}=6$).
It turns out that, 
by means of  $\Min(Q)$,  it is possible to compute the automorphism group of $Q$, 
and to compare the isomorphism class of $Q$ with  isomorphism classes of other Quebbemann lattices.
\par 
The purpose of this note is 
to generalize Quebbemann's construction slightly, and 
to show that this method indeed yields
 many mutually non-isomorphic extremal lattices  of rank $64$,
 by choosing the code $B$ (pseudo-)randomly, 
 and that their automorphism groups are often 
very  small.
Our main result below is proved by producing Quebbemann lattices $Q$ 
explicitly.
%
\begin{theorem}\label{thm:main}
Quebbemann's method yields
\begin{enumerate}[{\rm (1)}]
\item 
 at least $300$ isomorphism classes of  extremal lattices $Q$ of rank $64$ such that 
$\OG(Q)= \{\pm 1\}$, and 
\item 
at least $100$ isomorphism classes of extremal lattices $Q$ of rank $64$  such that  
$\OG(Q)\cong \{\pm 1\}\times \Z/8\Z$.
\end{enumerate}
\end{theorem}
See~Section~\ref{sec:Q} for the method to produce Quebbemann lattices,
Section~\ref{subsec:enumeration} for a method to  enumerate  minimal-norm vectors,
Section~\ref{subsec:isomclasses} for a method to distinguish isomorphism classes,
and Section~\ref{subsec:aut}  for the computation of the automorphism groups.
We exhibit a few examples in detail  in Section~\ref{sec:examples}.
The computation data of a part of the lattices in Theorem~\ref{thm:main}
is available from the author's web-page~\cite{ShimadaQ64compdata}.
(The whole data is too large to be put on a website.)
These data are written in the {\tt Record} format of  {\tt GAP}~\cite{GAP}.
\par 
In Chapter 8.3 (d) of~\cite{CSbook},
Conway and Sloane constructed 
a Quebbemann lattice that is different from the one given in
Quebbemann's original paper~\cite{Q84},
and suggested that there exist several isomorphism classes of Quebbemann lattices.
In~\cite{Q84Siegel}, Quebbemann showed by means of a mass formula
that there exist at least two isomorphism classes.
The ease with which we can make  non-isomorphic Quebbemann lattices
suggests that the number of isomorphism classes is very huge.
\par
Unimodular lattices with no non-trivial automorphisms
have been studied by many authors since the work of Bannai~\cite{Bannai1990}.
In~\cite{Mimura2006}, an even unimodular lattice of rank $64$ 
without non-trivial automorphisms 
is constructed.
This lattice is, however, not extremal.
\par
In~\cite{Nebe1998}, 
Nebe constructed
an extremal lattice of rank $64$ by a different method.
The order of the automorphism group is at least $587520$.
In~\cite{HKO2002} and~\cite{HM2014},
 the existence of extremal Type II $\Z_{2k}$-codes of length $64$ was shown.
The isomorphism classes and the automorphism groups of the associated 
extremal lattices are, however,  not clear.
In~\cite{Shimada2018}, 
another extremal lattice of rank $64$ was constructed
by means of a generalized quadratic residue code.
Its automorphism group  is of order $119040$.
%
\par
\medskip
Thanks are due to Professor Masaaki Harada 
and the anonymous referees  for their comments on the first version of this paper.
\par
\medskip
{\bf Conventions.}
Elements of a vector space or a lattice are written as 
\emph{row} vectors. 
For a lattice $L$  or a quadratic space $L$,
we denote by $\intfL{\, , \,}$ the symmetric bilinear form on $L$.
The automorphism group $\OG(L)$ of $(L, \intfL{\, , \,})$ acts on $L$ from the \emph{right}.

\section{Quebbemann lattice}\label{sec:Q}
\subsection{Quebbemann's construction}
We recall
Quebbemann's construction~\cite{Q84} of extremal lattices of rank $64$.
See also Chapter 8.3 (d) of~\cite{CSbook}.
In fact, 
our construction below  is slightly more general than Quebbemann's original.
\par
 Let $E$ be the root lattice of type $E_8$,
 that is, $E$ is the lattice of rank $8$ generated by vectors $e_1, \dots, e_8$ with 
 $\intfE{e_i, e_i}=2$ that form the dual graph in Figure~\ref{fig:E8}.
 It is well-known that $E$ is unimodular.
 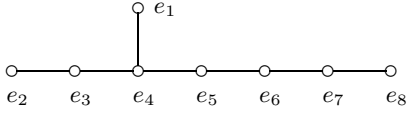
\begin{figure}
\def\ha{40}
\def\hav{37}
\def\hd{25}
\def\hdv{22}
\def\he{10}
\def\hev{7}
\setlength{\unitlength}{1.4mm}
{\small
\begin{picture}(80,12)(-11, 6)
\put(22, 16){\circle{1}}
\put(23.5, 15.5){$e\sb 1$}
\put(22, 10.5){\line(0,1){5}}
\put(9.5, \hev){$e\sb 2$}
\put(15.5, \hev){$e\sb 3$}
\put(21.5, \hev){$e\sb 4$}
\put(27.5, \hev){$e\sb 5$}
\put(33.5, \hev){$e\sb 6$}
\put(39.5, \hev){$e\sb 7$}
\put(45.5, \hev){$e\sb {8}$}
\put(10, \he){\circle{1}}
\put(16, \he){\circle{1}}
\put(22, \he){\circle{1}}
\put(28, \he){\circle{1}}
\put(34, \he){\circle{1}}
\put(40, \he){\circle{1}}
\put(46, \he){\circle{1}}
\put(10.5, \he){\line(5, 0){5}}
\put(16.5, \he){\line(5, 0){5}}
\put(22.5, \he){\line(5, 0){5}}
\put(28.5, \he){\line(5, 0){5}}
\put(34.5, \he){\line(5, 0){5}}
\put(40.5, \he){\line(5, 0){5}}
\end{picture}
}
\caption{Dynkin diagram of type $E_{8}$}\label{fig:E8}
\end{figure}
We consider the $\F_3$-quadratic space
$U:=E/3E$.
A subspace $V$ of $U$ is said  to be \emph{maximal isotropic}
if $\dim V=4$ and $\intfU{v, v}=0$ holds for all $v\in V$.
There exists a direct sum decomposition
\begin{equation}\label{eq:Udecomp}
U=V\oplus  W,  \quad \textrm{where $V$ and $ W$ are maximal isotropic subspaces.}
\end{equation}
Let $\DDD$ be the set of ordered pairs $(V, W)$ of maximal isotropic subspaces of $U$
satisfying $V \cap W=0$.
Since $\intfU{\,,\,}$  is non-degenerate,
we have a natural isomorphism
\begin{equation}\label{eq:dualisom}
 W\;\cong\; \Hom(V, \F_3)\quad \textrm{for each $(V, W)\in \DDD$}.
\end{equation}
Let $S$ denote the orthogonal direct sum $E^8$ of eight copies of $E$.
Then $S$ is unimodular of rank $64$.
For $i=1, \dots 8$,
we denote by $U_i$ the $i{\tthh}$-factor of  $S/3S=U^8$.
We  choose and fix an element
\[
\Delta:=(\, (V_1, W_1), \dots,  (V_8, W_8)\,)
\]
 of $\DDD^8$.
We put $T:=\{1, \dots, 8\}$,
and for a subset $J$ of $T$, 
we put
\[
U_J:=\bigoplus_{j\in J} U_j, \quad V_J:=\bigoplus_{j\in J} V_j, \quad W_J:=\bigoplus_{j\in J} W_j.
\]
Then we have
\[
S/3S\;=\;U_T\; = \; V_T\oplus W_T.
\]
We consider $U_T$, $V_T$ and $W_T$ as $\F_3$-vector spaces.
Let $B$ be a linear subspace of $V_T$
with $\dim B=8$.
Note that $ W_T$ can be regarded as  the dual space of $V_T$ by~\eqref{eq:dualisom}.
We put
\[
B\sperp:=\set{\barz\in W_T}{\intfUT{\barz, \bary}=0\;\;\textrm{for all}\;\; \bary\in B}.
\]
Then we have $\dim B\sperp=24$.
Let $\pi\colon S\to S/3S=U_T$ denote the natural projection.
For any elements $\barx, \barx\sprime$ of $B\oplus B\sperp\subset  U_T$,
we have $\intfUT{\barx, \barx\sprime}=0$,
and hence,
for any elements $x, x\sprime$ of $\pi\inv (B\oplus B\sperp)$,
we have $\intfS{x, x\sprime}\equiv 0 \bmod 3$.
We denote by $Q(\Delta, B)$
the lattice whose underlying $\Z$-module is $\pi\inv (B\oplus B\sperp)$
and whose bilinear form $\intfQ{\;,\;}$ is given by 
\[
\intfQ{\;,\;}:=\onethird \,\intfS{\;,\;}.
\]
Then $Q:=Q(\Delta, B)$ is an even lattice, and we have 
\[
\det Q= \left(\onethird\right)^{64}   \det S \cdot  [S:Q]^2
= \left(\onethird\right)^{64} \left(\frac{|U_T|}{|B\oplus B\sperp|}\right)^2
=1.
\]
\begin{definition}
For   $J=\{i, j\}\subset T$ with $|J|=2$,
we denote by $p_{J}\colon B\to V_J$
the projection to the $J$-factor.
We say that $B$ satisfies \emph{$\ptwo$-condition} 
if $p_J$ is an isomorphism for all  $J\subset T$ with $|J|=2$.
\end{definition}
\begin{remark}\label{rem:psix}
If the projection $p_J\colon B\to V_J$ is an isomorphism,
then the projection 
$p_{T\setminus J}\colon B\sperp\to  W_{T\setminus J}$ to the  $(T\setminus J)$-factor is also an isomorphism.
Hence, if $B$ satisfies $\ptwo$-condition, then 
$B\sperp$ satisfies  the following \emph{$\psix$-condition}:
for all   $J\sprime\subset T$ with $|J\sprime|=6$,
the projection $p_{J\sprime}\colon B\sperp\to  W_{J\sprime}$ to the $J\sprime$-factor is an isomorphism. 
\end{remark}
%
It is obvious that $\ptwo$-condition imposes an open condition
on the Grassmannian variety 
of $8$-dimensional subspaces $B$ of $V_T$.
%
%
\begin{proposition}[Quebbemann~\cite{Q84}]  \label{prop:Q}
Let $U_T=V_T\oplus  W_T$ be the  decomposition of $U_T=S/3S$
associated with an element $\Delta$ of $ \DDD^8$, 
and let $B$ be an $8$-dimensional subspace of $V_T$.
Suppose that $B$ satisfies $\ptwo$-condition.
Then $\min (Q(\Delta, B))=6$ holds, that is, 
 $Q(\Delta, B)$ is an extremal lattice of rank $64$.
\end{proposition}
\begin{proof}
We write $x\in Q(\Delta, B)$  as $x=(x_1, \dots, x_8)$, 
where  $x_i\in E$ is the $i\tthh$ component
by the embedding $Q(\Delta, B)\inj S=E^8$.
We put
\[
\barx:=\pi(x)=(\barx_1, \dots, \barx_8)\;\in\; B\oplus B\sperp\;\subset\; U_T,
\]
where $\barx_i\in U_i$ is $x_i \bmod 3E$.
Decomposing each  $\barx_i$ as 
 $\bary_i+\barz_i$
with  $\bary_i\in V_i$ and $\barz_i\in  W_i$,
we obtain $\barx=\bary+\barz$, where 
\[
\bary:=(\bary_1, \dots, \bary_8)\in B,
\quad 
\barz:=(\barz_1, \dots, \barz_8)\in B\sperp.
\]
Suppose that 
 $\intfQ{x, x}\le 4$.
We show that $x=0$.
Since 
\begin{equation}\label{eq:sum}
\intfS{x, x}=\sum \intfE{x_i, x_i}\le 12,
\end{equation}
we see that at least two of the components $x_i$ are zero.
Since at least two of  $\bary_i$ are zero, the assumption that
$B$ satisfy $\ptwo$-condition implies $\bary=0$.
Therefore we have $\barx=\barz$.
In particular, each $\barx_i$ belongs to $ W_i$.
Since $ W_i$ is isotropic in $U=E/3E$, 
we have $\intfE{x_i, x_i}\equiv 0\bmod 3$
and hence $\intfE{x_i, x_i}\equiv 0\bmod 6$.
Combining this with~\eqref{eq:sum}, 
we see that at most two of $x_i$ are non-zero.
Since $B\sperp$ satisfies $\psix$-condition by Remark~\ref{rem:psix}, we see that 
$\barz=0$. Therefore $\barx=0$,
and hence $x\in 3S$.
 If $x$ were non-zero,
we would have $\intfS{x,x}\ge 18$,
which is a contradiction.
\end{proof}
\begin{definition}
An extremal lattice of rank $64$ of the form $Q(\Delta, B)$, 
where $\Delta$ is an element of $\DDD^8$ and 
$B$ is an $8$-dimensional subspace of  $V_T$ satisfying $\ptwo$-condition,  
is called a \emph{Quebbemann lattice}.
\end{definition}
\subsection{Maximal isotropic subspaces of $U$}\label{subsec:UV1V2}
Recall that the lattice $E$ is equipped with a basis $e_1, \dots, e_8$ in Figure~\ref{fig:E8}.
We write elements of $E$ or of  $U=E/3E$ as row vectors 
with respect to $e_1, \dots, e_8$.
The automorphism group $\OG(E)$ 
of $E$ is generated by the reflections 
\[
x\;\mapsto\;  x\,-\,\intfE{x, e_i}\;  e_i
\]
with respect to the  vectors  $e_i$ ($i=1, \dots, 8$) of norm $2$,
and is of order $2^{14}\cdot  3^{5} \cdot 5^2\cdot 7$.
The natural homomorphism
$\OG(E)\to \OG(U)$
to the automorphism group of the $\F_3$-quadratic space $U$ is 
injective.
Let $\VVV$ be the set of maximal isotropic subspaces $V$ of $U$.
We can prove the following by 
the standard orbit stabilizer algorithm 
using {\tt GAP}~\cite{GAP}.
\begin{proposition}\label{prop1:UV1V2}
The size of $\VVV$ is $2240$, and 
$\OG(E)$ acts  transitively on $\VVV$. 
\qed
\end{proposition}
Let $V_0\in \VVV$ be the maximal isotropic subspace with  basis $v_1, \dots, v_4$ 
in Table~\ref{table:basisvs}.
The stabilizer subgroup $\Stab(V_0)$ of $V_0$ in $\OG(E)$ is of order $2^8\cdot 3^5 \cdot 5$.
Let $\WWW(V_0)$ be the set of all $W\in \VVV$
such that $V_0 \cap W=0$.
We have $|\WWW(V_0)|=729$.
\begin{proposition}\label{prop2:UV1V2}
The action of $\Stab(V_0)$ decomposes $\WWW(V_0)$ 
into two orbits
of size $648$ and $81$.
The orbit of size $648$ contains $W^\typeI$
with  basis $v_1^{*\typeI}, \dots, v_4^{*\typeI}$ 
in Table~\ref{table:basisvs},
and the orbit of size $81$ contains $W^{\typeII}$
with  basis $v_1^{*\typeII}, \dots, v_4^{*\typeII}$ 
in Table~\ref{table:basisvs}.
\qed
\end{proposition}
%
%
%
%
\begin{table}
\[
\begin{array}{lcl}
v_1= ( 0, 0, 0, 0, 0, 0, 1, 2 ),   &  \phantom{aa}  & v_2=( 0, 0, 0, 1, 2, 0, 0, 0  ), \\
v_3=( 0, 1, 0, 0, 1, 0, 0, 1  ), & & v_4= ( 1, 0, 2, 0, 0, 0, 0, 1  ).
\end{array}
\]
\vskip 0mm
\[
\begin{array}{lcl}
v_1^{*\typeI}= ( 1, 2, 1, 0, 2, 2, 0, 2  ),   &  \phantom{aa}  & v_2^{*\typeI}= ( 0, 1, 2, 0, 0, 0, 0, 0   ), \\
v_3^{*\typeI}=( 2, 2, 1, 0, 1, 0, 2, 2  ), & & v_4^{*\typeI}= ( 1, 0, 2, 0, 1, 0, 2, 1  ).
\end{array}
\]
\vskip 0mm
\[
\begin{array}{lcl}
v_1^{*\typeII}= ( 1, 2, 1, 0, 2, 2, 0, 2  ),   &  \phantom{aa}  & v_2^{*\typeII}= ( 1, 1, 1, 0, 0, 0, 0, 1  ), \\
v_3^{*\typeII}=( 0, 2, 0, 0, 1, 0, 2, 0  ), & & v_4^{*\typeII}= ( 1, 2, 2, 2, 1, 0, 2, 0  ).
\end{array}
\]
\vskip 4mm
\caption{Bases of maximal isotropic subspaces $V_0, W^{\typeI}, W^{\typeII}$}\label{table:basisvs}
\end{table}
\begin{remark}
The basis $v_1^*, \dots, v_4^*$ of $ W$ above is dual to 
the basis $v_1, \dots, v_4$ of $V_0$ by the canonical pairing~\eqref{eq:dualisom}. 
\end{remark}
\begin{corollary}\label{cor:GIGII}
The action of $\OG(E)$ decomposes $\DDD$ into two orbits.
One orbit contains $(V_0, W^{ \typeI})$ with the stabilizer subgroup $G^{\typeI}$ of order $480$,
and the other orbit contains $(V_0, W^{  \typeII})$ with the stabilizer subgroup $G^{\typeII}$  of order $3840$.
\qed
\end{corollary}
\erase{
The group $\OG(E)^8\semidirectproduct \SSSS_8$ acts on $S$,
where $\SSSS_8$ is the symmetric  group of degree $8$ acting on $S=E^8$ by permutations of factors.
Considering this action,
we obtain the following:
\begin{corollary}
For $k=0, \dots, 8$, let $\Delta_k$ denote  the element
$((V_1, W_1), \dots, (V_8, W_8))$ of $\DDD^8$ such that 
$V_i$ is equal to the representative $V_0$ for $i=1, \dots, 8$ and that
\[
W_i=\begin{cases}
W^{\typeI} & \textrm{for $i\le k$,} \\
W^{\typeII} & \textrm{for $i>k$.} 
\end{cases}
\]
Then
each  Quebbemann lattice 
is isomorphic to a  Quebbemann lattice 
of the form $Q(\Delta_k, B)$ for some $k$,
where $B$ is an $8$-dimensional linear subspace of $V_0^8$ satisfying $\ptwo$-condition.
\qed
\end{corollary}
}
\subsection{Construction of $Q(\Delta, B)$ with an automorphism  of order $8$}\label{subsec:cyc8}
We fix  a pair $(V,  W)\in \DDD$,
and consider the   $8$-tuple
\[
\Delta_0:=((V,  W), \dots, (V,  W))\in \DDD^8.
\]
Let $G$ be the stabilizer subgroup of $(V,  W)$ in $\OG(E)$,
and 
let $\gamma$ be an element of order $8$ in $G$.
(The stabilizer subgroup $G^{\typeI}$ (resp.~$G^{\typeII}$) in Corollary~\ref{cor:GIGII} contains 
$120$ elements (resp.~$1360$ elements)  of order $8$.)
We define an automorphism  $\tilde{\gamma}$ of $S=E^8$ by
\[
x=(x_1, \dots, x_8)\mapsto x^{\tilde{\gamma}}=(x_2^\gamma,  \dots, x_8^\gamma, x_1^\gamma).
\]
The action of $\tilde{\gamma}$ on $S/3S=U_T$ preserves the decomposition $U_T=V_T\oplus  W_T$
associated with $\Delta_0$ above.
For $v\in V_T=V^8$,
we denote by $B(\gamma, v)$ the linear subspace of $V_T$
generated by the orbit of $v$ under the action of  $\gen{\tilde{\gamma}}\cong \Z/8\Z$.
If $B(\gamma, v)$ is of dimension $8$ and satisfies $\ptwo$-condition,
then $Q(\Delta_0, B(\gamma, v))$ is a Quebbemann lattice 
invariant under the action of $\gen{\tilde{\gamma}}$ on $S$.
In particular, the automorphism group $\OG(Q)$ of $Q:=Q(\Delta_0, B(\gamma, v))$ contains 
an element 
\[
\tilde{\gamma}_Q:=\tilde{\gamma}|Q
\]
 of order $8$.
\section{Computations on Quebbemann lattices}
We fix an $8$-tuple   $\Delta=((V_1, W_1), \dots, (V_8, W_8))\in \DDD^8$.
Let $B$ be an $8$-dimensional linear subspace of $V_T=V_1\oplus\cdots\oplus V_8$
satisfying $\ptwo$-condition,
and we consider the extremal lattice $Q(\Delta, B)$ of rank $64$.
\subsection{Enumeration of minimal-norm vectors}\label{subsec:enumeration}
In this section,
we explain a method to calculate the set  
\[
\Min(Q(\Delta, B)):=\set{x\in Q(\Delta, B)}{\intfQ{x, x}=6}=\set{x\in Q(\Delta, B)}{\intfS{x, x}=18}
\]
of all minimal-norm vectors.
A \emph{norm-type} is an $8$-tuple $[n_1, \dots, n_8]$ 
of non-negative even integers $n_i$ such that $\sum n_i=18$.
For $x=(x_1, \dots, x_8)\in \Min(Q(\Delta, B))$, 
we put
\[
\nu(x):=[\intfE{x_1, x_1}, \dots, \intfE{x_8, x_8}],
\]
and call it the \emph{norm-type of  $x$}.
For a non-negative even integer $a$, we put
\[
N(a, E):=\set{v\in E}{\intfE{v,v}=a},
\]
and let $N(a, U)\subset U$ be the image of $N(a, E)$ by 
the natural projection $E\to U$.
(See~Table~\ref{table:NaENaU}.)
A minimal-norm vector $x\in \Min(Q(\Delta, B))$ is said to be \emph{of divisible type}
if $x\in 3S$ holds,
or equivalently,
if  only one of $x_1, \dots, x_8$ (say $x_i$) is non-zero
and $x_i$ is written as $3 x\sprime_i$ by some $x\sprime_i\in N(2, E)$,
or equivalently,
if the norm-type $\nu(x)$ of $x$ is obtained by a permutation of components from 
$[0,\dots, 0, 18]$.
Since $ |N(2, E)|=240$, 
there exist exactly $240\times 8$ minimal-norm vectors of divisible type.
\begin{table}
\[
\begin{array}{c | cc}
\;\;a\;\;  & |N(a, E)| & |N(a, U)| \\
\hline 
 0 & 1 & 1 \\
 2 & 240 & 240 \\
 4 & 2160 & 2160\\
 6 & 6720 & 2240
\end{array}
\]
\caption{$N(a, E)$ and $N(a, U)$}\label{table:NaENaU}
\end{table}
\begin{proposition}
Let $x\in Q(\Delta, B)$ be a minimal-norm vector that is not of divisible type.
Then the norm-type $\nu (x)$ of $x$ is obtained by a permutation of components from  one of the following:
\begin{equation}\label{eq:normtypes}
\begin{array}{ll}
{[0,0,0,0,0,6,6,6]}  & \textrm{\rm  (of type $0^5 6^3$)}, \\
{[0,2,2,2,2,2,4,4]}  & \textrm{\rm (of type $0^1 2^5 4^2$)}, \\
{[0,2,2,2,2,2,2,6]}  & \textrm{\rm (of type $0^1 2^6 6^1$)}, \\
{[2,2,2,2,2,2,2,4]}  & \textrm{\rm (of type $2^7 4^1$)}.
\end{array}
\end{equation}
\end{proposition}
\begin{proof}
As in the proof of~Proposition~\ref{prop:Q}, 
we see that  $\barx:=\pi(x)\in B\oplus B\sperp$ 
is decomposed uniquely as $\bary+\barz$, where 
$\bary=(\bary_1, \dots, \bary_8)\in B$ and $\barz=(\barz_1, \dots, \barz_8)\in B\sperp$.
Suppose that $\bary=0$.
Since $x$ is not of divisible type,
we see that $\barx=\barz$ is not zero.
Since $B\sperp$ satisfies $\psix$-condition,
at most five of $\barz_1, \dots, \barz_8$  are zero.
Since $\barx_i=\barz_i\in  W_i$, we have $\intfE{x_i, x_i}\equiv 0 \bmod 3$ and hence 
$\intfE{x_i, x_i} \in \{0,6,12, 18\}$.
Combining these, we see that $\nu(x)$ is of type $0^5 6^3$.
Suppose that $\bary\ne0$.
Since $B$ satisfies $\ptwo$-condition,
at most one of $\bary_1, \dots, \bary_8$ is zero.
Hence at most one of $x_1, \dots, x_8$ is zero.
Therefore $\nu(x)$ is either  of type $0^1 2^5 4^2$ or $0^1 2^6 6^1$  or $2^7 4^1$.
\end{proof}
For $k=1, \dots, 8$,
let $e_1\spar{k}, \dots, e_8\spar{k}$ be the copy of the basis 
$e_1, \dots, e_8$ of $E$ in the $k\tthh$ factor of $S=E^8$.
We use the ordered set
\begin{equation}
e_1\spar{1}, \dots, e_8\spar{1}, e_1\spar{2}, \dots, e_8\spar{2},
\dots\dots\dots, 
e_1\spar{8}, \dots, e_8\spar{8}
\end{equation}
of vectors 
as a basis of $S$ and of $S/3S=U^8=U_T$.
\begin{proposition}
The ternary code $B\oplus B\sperp\subset U_T$ 
is generated by row vectors 
of a $32\times 64$ matrix of the echelon form as in Figure~\ref{fig:echelon},
where $I_8$ is the identity matrix of size $8$,
$A_i$ are $4\times 8$ matrices whose row vectors form a basis of $ W_i\subset U$
for $i=3, \dots, 6$,
 $C_{\mu\nu}$ are some $8\times 8$ matrices,
 and the blank blocks are zero matrices.
\end{proposition}
\begin{proof}
Since the projection $p_{12}\colon B\to V_1\oplus V_2$ to the $(12)$-factor 
and  the projection  $p_{\,\overline{78}}\colon B\sperp\to  W_1\oplus\cdots\oplus W_6$ to the $(123456)$-factor
are both isomorphisms, the projection 
\[
P_{12}\colon B\oplus B\sperp \to U\oplus U
\]
to the $(12)$-factor  is surjective, and its kernel $\Ker P_{12}$
is mapped isomorphically to $ W_3\oplus\cdots\oplus W_6$ 
by the projection 
\[
P_{3456}\colon \Ker P_{12}\to U\oplus U\oplus U \oplus U
\]
to the $(3456)$-factor.
\end{proof}
\begin{figure}
\setlength{\unitlength}{1.4mm}
\begin{picture}(74, 38)(-10,-1)
\put(0,0){\line(1,0){64}}
\put(0,8){\line(1,0){64}}
\put(0,16){\line(1,0){64}}
\put(0,24){\line(1,0){64}}
\put(0,32){\line(1,0){64}}
\put(0,0){\line(0,1){32}}
\put(8,0){\line(0,1){32}}
\put(16,0){\line(0,1){32}}
\put(24,0){\line(0,1){32}}
\put(32,0){\line(0,1){32}}
\put(40,0){\line(0,1){32}}
\put(48, 0){\line(0,1){32}}
\put(56, 0){\line(0,1){32}}
\put(64, 0){\line(0,1){32}}
\put(16,12){\line(1,0){16}}
\put(32,4){\line(1,0){16}}
\put(3.2,27.2){$I_8$}
\put(11.2,19.2){$I_8$}
\put(19.0,13.2){$A_3$}
\put(27.0,9.2){$A_4$}
\put(35.0,5.2){$A_5$}
\put(43.0,1.2){$A_6$}
\put(18.4,27.2){$C_{13}$}
\put(26.4,27.2){$C_{14}$}
\put(34.4,27.2){$C_{15}$}
\put(42.4,27.2){$C_{16}$}
\put(50.4,27.2){$C_{17}$}
\put(58.4,27.2){$C_{18}$}
\put(18.4,19.2){$C_{23}$}
\put(26.4,19.2){$C_{24}$}
\put(34.4,19.2){$C_{25}$}
\put(42.4,19.2){$C_{26}$}
\put(50.4,19.2){$C_{27}$}
\put(58.4,19.2){$C_{28}$}
\put(50.4,11.2){$C_{37}$}
\put(58.4,11.2){$C_{38}$}
\put(50.4,3.2){$C_{47}$}
\put(58.4,3.2){$C_{48}$}
\end{picture}
\caption{Echelon form of a generator matrix of $B\oplus B\sperp$}\label{fig:echelon}
\end{figure}
We make the symmetric group $\SSSS_8$ act on $S=E^8$ by
\[
(x_1, \dots, x_8)^{\sigma}:=(x_{\sigma(1)}, \dots, x_{\sigma(8)})
\quad \textrm{for $\sigma \in \SSSS_8$.}
\]
For $\Delta=((V_1, W_1), \dots, (V_8, W_8)) \in \DDD^8$,
we put
\[
\Delta^{\sigma}:=((V_{\sigma(1)}, W_{\sigma(1)}), \dots, (V_{\sigma(8)}, W_{\sigma(8)})).
\]
Then 
we have $Q(\Delta, B)^{\sigma}=Q(\Delta^{\sigma}, B^{\sigma})$ in $S$.
If $x\in \Min(Q(\Delta, B))$ is of norm-type $[n_1, \dots, n_8]$,
then $x\sp{\sigma}\in \Min(Q(\Delta^{\sigma}, B\sp{\sigma}))$ is of norm-type $[n_{\sigma(1)}, \dots, n_{\sigma(8)}]$.
\par
Let $n=[n_1, \dots, n_8]$ be a norm-type 
obtained by a permutation of components from one of 
the norm-types in~\eqref{eq:normtypes}.
We calculate the set $\barM(n)$ 
of codewords $\barx=\pi(x)\in B\oplus B\sperp$ 
corresponding $x\in \Min(Q(\Delta, B))$ with $\nu(x)=n$
by the following method.
\par
First 
we choose a permutation $\tau\in \SSSS_8$ such that  $n^{\tau}=[n_{\tau(1)}, \dots, n_{\tau(8)}]$ satisfies 
\[
n_{\tau(1)}\;\le\; n_{\tau(2)}\;\le\;\dots\; \le\; n_{\tau(8)}.
\]
We then transform a generator matrix of $(B\oplus B\sperp)^{\tau}=B^{\tau}\oplus B^{\tau\perp}$ 
into the echelon form in Figure~\ref{fig:echelon},
and search for $\barx_1, \dots, \barx_8\in U$ satisfying
conditions~\eqref{eq:barx1barx1} below
by back track search;
that is, 
if we find $(\barx_1, \dots, \barx_i)$ satisfying the first $i$ conditions of~\eqref{eq:barx1barx1},
then we search for $\barx_{i+1}$  satisfying the $(i+1)\sstt$  condition of~\eqref{eq:barx1barx1}.
Recall that  $N(n_i, U)$ is the image of $N(n_i, E)$ by the natural map $E\to U$.
\begin{equation}
\renewcommand{\arraystretch}{1.4}%
\begin{array}{ll}
 \barx_1  \in N(n_{\tau(1)},  U),  \\
 \barx_2 \in N(n_{\tau(2)}, U),   \\
 \barx_3 :=  \barx_1 C_{13} + \barx_2 C_{23}+ \baru_3 A_3 \in N(n_{\tau(3)}, U), \;\;\textrm{where $\baru_3\in  \F_3^4$},  \\
 \barx_4 :=  \barx_1 C_{14} + \barx_2 C_{24}+ \baru_4 A_4 \in N(n_{\tau(4)}, U), \;\;\textrm{where $\baru_4\in  \F_3^4$},  \\
 \barx_5 :=  \barx_1 C_{15} + \barx_2 C_{25}+ \baru_5 A_5 \in N(n_{\tau(5)}, U), \;\;\textrm{where $\baru_5\in  \F_3^4$},  \\
 \barx_6 :=  \barx_1 C_{16} + \barx_2 C_{26}+ \baru_6 A_6 \in N(n_{\tau(6)}, U), \;\;\textrm{where $\baru_6\in  \F_3^4$},  \\
 \barx_7:=  \barx_1 C_{17} + \barx_2 C_{27}+  (\baru_3, \baru_4)C_{37}+ (\baru_5, \baru_6) C_{47} \in N(n_{\tau(7)}, U), \\
 \barx_8:= \barx_1 C_{18} + \barx_2 C_{28}+  (\baru_3, \baru_4) C_{38}+ (\baru_5, \baru_6) C_{48} \in N(n_{\tau(8)}, U).
\end{array}
\label{eq:barx1barx1}
\end{equation}
If we find $\barx=(\barx_1, \dots, \barx_8)$
satisfying
all conditions in~\eqref{eq:barx1barx1},
then $\barx$ belongs to $\barM(n^{\tau})$ and hence
\[
\barx^{\tau\inv}=(\barx_{\tau\inv(1)}, \dots, \barx_{\tau\inv(8)})
\] 
is an element of $\barM(n)$.
All elements of $\barM(n)$ are obtained in this way.
\par
Using the maps $N(a, E)\to N(a, U)$,
we can make from $\barM(n)$ the set $M(n)$ 
of  vectors  $x\in \Min(Q(\Delta, B))$ with $\nu(x)=n$.
Taking the union of these sets $M(n)$ 
together with the set of minimal-norm vectors of divisible type,
we obtain the set $\Min(Q(\Delta, B))$ of all minimal-norm vectors of $Q(\Delta, B)$.
\begin{remark}
Thanks to the permutation $\tau$, 
we have 
$|N(n_{\tau(i)}, U)|\le |N(n_{\tau(j)}, U)|$
for $i<j$, 
 and hence, in  the back track search above,
there exist few possibilities of $\barx_i$ for small indexes $i$.
By this trick, 
the enumeration of $\Min(Q(\Delta, B))$ becomes tractable.
\end{remark}
\begin{remark}
We know that $|\Min(L)|=2611200$ for an extremal lattice $L$ of rank $64$ 
by the theory of theta series and modular forms.
(See, for example, Chapter 7.7 of~\cite{CSbook}.)
Hence we can confirm easily that we left no minimal-norm vectors uncounted.
\end{remark}
\subsection{Isomorphism classes}\label{subsec:isomclasses}
In order to distinguish isomorphism classes of two extremal lattices 
$L$ and $L\sprime$ of rank $64$,
we use the \emph{distribution of intersection patterns of minimal-norm vectors}.
Let $\Min(L)$ be the set of vectors $v\in L$ with $\intfL{v, v}=6$.
For $v, v\sprime \in \Min(L)$, we have
$\intfL{v, v\sprime}\in \{0,\pm1, \pm 2,\pm 3, \pm 6\}$.
For $k=0,1,2,3,6$, we put
\[
a_k (v):=\frac{\,1\,}{2}\; |\,\set{v\sprime\in \Min(L)}{ \intfL{v, v\sprime}=k \;\textrm{or}\;-k }\,|.
\]
We have $a_6(v)=1$ and $\sum a_k(v)=1305600$.
The triple $a(v):=[a_1(v), a_2(v), a_3(v)]$ is called the \emph{intersection pattern} of $v$.
For a triple $a=[a_1, a_2, a_3]$ of non-negative integers with $a_1+a_2+a_3+1\le  1305600$,
we put
\[
\AAA_L(a):=\set{v\in \Min(L)}{a(v)=a},
\quad 
A_L(a):=\frac{\,1\,}{2}\; |\AAA_L(a) |,
\]
and call the function $A_L$ the \emph{distribution of intersection patterns}.
It is  obvious that, if $A_L\ne A_{L\sprime}$, then  
$L$ and $L\sprime$  are not isomorphic.
\begin{remark}
In fact, the calculation of intersection patterns $a(v)$ 
of all  elements $v$ of $\Min(Q(\Delta, B))/\{\pm 1\}$ takes 
most of the computation time
in the proof of Theorem~\ref{thm:main}.
\end{remark}
\subsection{Automorphism group}\label{subsec:aut}
Let $L$ be an extremal lattice of rank $64$, and 
let $\Gamma$ be a subgroup of $\OG(L)$.
 We will apply Proposition~\ref{prop:rig} below to 
$\Gamma=\{\pm 1\}$ or $\Gamma=\{\pm 1\}\times \gen{\tilde{\gamma}_Q}$,
where $\tilde{\gamma}_Q$ is the automorphism of $Q(\Delta_0, B(\gamma, v))$ introduced in Section~\ref{subsec:cyc8}.

\begin{definition}
An ordered list $(v_1, \dots, v_{64})$ of vectors in $\Min(L)$ is said to be a \emph{$\Gamma$-rigidifying basis}
if the following hold:
\begin{enumerate}[(a)]
\item The vectors $v_1, \dots, v_{64}$ form a basis of $L\tensor\Q$.
\item The group $\Gamma$ acts on the set $\AAA_L(a(v_1))$ transitively.
\item Suppose that $i>1$.
Then the set
\[
\set{v\sprime\in \AAA_L(a(v_i))}{\intfL{v\sprime, v_j}=\intfL{v_i, v_j}\;\;\textrm{for all}\;\; j<i\;}
\]
consists of a single element $v_i$.
\end{enumerate}
\end{definition}
\begin{proposition}\label{prop:rig}
If a $\Gamma$-rigidifying basis exists, then $\OG(L)$ coincides with $\Gamma$.
\end{proposition}
\begin{proof}
Note that $\OG(L)$ preserves each subset set $\AAA_L(a)$ of $\Min(L)$ for any $a$.
In particular, we have $v^g\in \AAA_L(a(v))$  for any $g\in \OG(L)$ and any $v\in L$.
Let $g$ be an arbitrary element of $\OG(L)$.
By (b), there exists an element $g\sprime\in \Gamma$ 
such that $v_1^g=v_1^{g\sprime}$.
By (c), we can prove that $v_i^g=v_i^{g\sprime}$ holds for all $i=1, \dots, 64$ 
by induction on $i$.
Then (a) implies that $g=g\sprime$.
\end{proof}
\section{Examples}\label{sec:examples}
\subsection{Examples without non-trivial automorphisms}
 \begin{table}
{
\footnotesize
\[
\setlength{\arraycolsep}{4pt}
\left[\begin{array}{cccc|cccc|cccc|cccc|cccc|cccc} 
0 & 0 & 2 & 0 & 1 & 0 & 0 & 1 & 1 & 0 & 2 & 2 & 1 & 0 & 0 & 1 & 1 & 0 & 2 & 1 & 0 & 2 & 2 & 2 \\ 
1 & 1 & 0 & 2 & 2 & 0 & 0 & 1 & 0 & 1 & 0 & 1 & 1 & 0 & 0 & 2 & 0 & 1 & 1 & 0 & 2 & 0 & 1 & 2 \\ 
1 & 0 & 2 & 0 & 0 & 2 & 0 & 0 & 2 & 1 & 0 & 1 & 1 & 0 & 1 & 1 & 1 & 2 & 1 & 2 & 0 & 2 & 1 & 1 \\ 
2 & 1 & 0 & 0 & 1 & 1 & 1 & 1 & 1 & 0 & 1 & 0 & 2 & 1 & 0 & 0 & 1 & 2 & 0 & 0 & 0 & 2 & 2 & 0 \\ 
\hline
1 & 1 & 0 & 0 & 2 & 0 & 1 & 2 & 1 & 0 & 2 & 2 & 0 & 0 & 2 & 2 & 2 & 1 & 0 & 1 & 1 & 1 & 2 & 2 \\ 
2 & 0 & 1 & 0 & 2 & 0 & 0 & 1 & 1 & 0 & 0 & 2 & 0 & 2 & 0 & 0 & 1 & 1 & 1 & 0 & 1 & 0 & 1 & 2 \\ 
1 & 0 & 2 & 1 & 2 & 2 & 2 & 1 & 1 & 2 & 2 & 2 & 1 & 1 & 1 & 2 & 1 & 1 & 1 & 2 & 1 & 2 & 2 & 1 \\ 
2 & 2 & 1 & 1 & 2 & 0 & 2 & 2 & 1 & 0 & 2 & 1 & 1 & 0 & 1 & 1 & 1 & 1 & 2 & 1 & 0 & 0 & 2 & 0 
\end{array}\right] 
\]
}
\caption{Matrix $G\sprime_0$}\label{table:Gsprime0}
\end{table}
%
 %
 %
\begin{table}
{
\footnotesize
\[
\begin{array}{c|ccc|c} 
\textrm{no.} & a_1 & a_2 & a_3 & A_{Q^\typeI} (a) \\
\hline
1 & 568092 & 40191 & 612 & 1 \\ 
2 & 568290 & 40155 & 606 & 1 \\ 
3 & 568356 & 40143 & 604 & 3 \\ 
4 & 568488 & 40119 & 600 & 2 \\ 
5 & 568554 & 40107 & 598 & 2 \\ 
6 & 568620 & 40095 & 596 & 3 \\ 
7 & 568686 & 40083 & 594 & 2 \\ 
8 & 568752 & 40071 & 592 & 6 \\ 
9 & 568818 & 40059 & 590 & 6 \\ 
10 & 568884 & 40047 & 588 & 9 \\ 
&\dots &&&\\
110 & 579840 & 38055 & 256 & 5333 \\ 
111 & 579906 & 38043 & 254 & 6275 \\ 
112 & 579972 & 38031 & 252 & 7616 \\ 
113 & 580038 & 38019 & 250 & 8752 \\ 
114 & 580104 & 38007 & 248 & 10065 \\ 
115 & 580170 & 37995 & 246 & 11511 \\ 
116 & 580236 & 37983 & 244 & 13332 \\ 
117 & 580302 & 37971 & 242 & 15370 \\ 
118 & 580368 & 37959 & 240 & 17252 \\ 
119 & 580434 & 37947 & 238 & 19533 \\ 
120 & 580500 & 37935 & 236 & 22294 \\ 
&\dots &&&\\
170 & 583800 & 37335 & 136 & 17 \\ 
171 & 583866 & 37323 & 134 & 9 \\ 
172 & 583932 & 37311 & 132 & 6 \\ 
173 & 583998 & 37299 & 130 & 5 \\ 
174 & 584130 & 37275 & 126 & 3 \\ 
175 & 584196 & 37263 & 124 & 2 \\ 
176 & 584262 & 37251 & 122 & 1 \\
\hline
&&&\llap{total} & 1305600
\end{array}
%
\phantom{aaaaaa}
%
\begin{array}{c|ccc|c} 
\textrm{no.} & a_1 & a_2 & a_3 & A_{Q^{\typeII}} (a) \\
\hline
1 & 568422 & 40131 & 602 & 1 \\ 
2 & 568488 & 40119 & 600 & 2 \\ 
3 & 568554 & 40107 & 598 & 1 \\ 
4 & 568620 & 40095 & 596 & 3 \\ 
5 & 568686 & 40083 & 594 & 7 \\ 
6 & 568752 & 40071 & 592 & 5 \\ 
7 & 568818 & 40059 & 590 & 8 \\ 
8 & 568884 & 40047 & 588 & 8 \\ 
9 & 568950 & 40035 & 586 & 11 \\ 
10 & 569016 & 40023 & 584 & 11 \\ 
&\dots &&&\\
110 & 580104 & 38007 & 248 & 9761 \\ 
111 & 580170 & 37995 & 246 & 11289 \\ 
112 & 580236 & 37983 & 244 & 13121 \\ 
113 & 580302 & 37971 & 242 & 15330 \\ 
114 & 580368 & 37959 & 240 & 17148 \\ 
115 & 580434 & 37947 & 238 & 19598 \\ 
116 & 580500 & 37935 & 236 & 22119 \\ 
117 & 580566 & 37923 & 234 & 24532 \\ 
118 & 580632 & 37911 & 232 & 27067 \\ 
119 & 580698 & 37899 & 230 & 29774 \\ 
120 & 580764 & 37887 & 228 & 32471 \\ 
&\dots &&&\\
170 & 584064 & 37287 & 128 & 2 \\ 
171 & 584130 & 37275 & 126 & 1 \\ 
172 & 584196 & 37263 & 124 & 2 \\ 
173 & 584262 & 37251 & 122 & 1 \\ 
174 & 584328 & 37239 & 120 & 1 \\
&&&& \\
&&&&\\
\hline
&&& \llap{total}& 1305600
\end{array}
\]
}
\vskip .5cm
\caption{Distributions of intersection patterns of $Q^{\typeI}$ and $Q^{\typeII}$}\label{table:AQB0s}
\end{table}
Let $V_0$ be the maximal isotropic subspace of $U$ with basis $v_1, \dots, v_4$ 
in Table~\ref{table:basisvs}.
By this basis,
an element of $V_0$ is expressed by a vector in $\F_3^4$, and hence 
an element of $V_0^8$ is expressed by a vector in $\F_3^{32}$.
Let $G_0$ be the $8\times 32$  matrix with components in $\F_3$ 
of the form $[\,I_8 \,|\, G\sprime_0\,]$,
where $I_8$ is the identity matrix of size $8$ and $G\sprime_0$ is given in Table~\ref{table:Gsprime0}.
(This matrix  $G\sprime_0$ was produced by choosing  components pseudo-randomly.)
Let $B_0$ be the linear subspace of $V_0^8$ generated by the row vectors of $G_0$.
Then $B_0$ satisfies $\ptwo$-condition.
Recall that we have given maximal isotropic subspaces $W^{\typeI}$ and $W^{\typeII}$  in Table~\ref{table:basisvs}.
Let $Q^\typeI$ (resp.~$Q^\typeII$ ) be the Quebbemann lattice
$Q(\Delta\sp{\typeI}, B_0)$
(resp.~$Q(\Delta\sp{\typeII}, B_0)$),
where 
\[
\Delta\sp{\typeI}=((V_0, W^{\typeI}), \dots, (V_0, W^{\typeI})) \in \DDD^8,
\quad
\Delta\sp{\typeII}=((V_0, W^{\typeII}), \dots, (V_0, W^{\typeII})) \in \DDD^8.
\]
Then the distributions  of intersection patterns of $Q^\typeI$ and  $Q^\typeII$
are as in Table~\ref{table:AQB0s}.
(The left table is of $Q^\typeI$ and the right is of $Q^\typeII$.)
In these tables, 
the intersection patterns $a=[a_1, a_2, a_3]$ are sorted by the lexicographic order.
We can readily see that  
$Q^\typeI$ and $Q^\typeII$ are not isomorphic.
Both of $Q^\typeI$ and $Q^\typeII$
have $\{\pm 1\}$-rigidifying basis,
and hence $\OG(Q^\typeI)$ and $\OG(Q^\typeII)$
are equal to $\{\pm 1\}$.
 %
 %
\subsection{Examples with an automorphism of order $8$}
%
%
%
\begin{table}
{
\footnotesize
\[
 \setlength\arraycolsep{2.5pt} 
\gamma:=
\left[\begin{array}{cccccccc} 
2 & 1 & 2 & 4 & 3 & 2 & 2 & 1 \\ 
1 & 1 & 1 & 1 & 1 & 1 & 1 & 1 \\ 
-1 & -1 & -2 & -2 & -2 & -2 & -2 & -1 \\ 
-1 & 0 & 0 & -1 & 0 & 0 & 0 & 0 \\ 
0 & -1 & -1 & -1 & -1 & 0 & 0 & 0 \\ 
2 & 2 & 3 & 4 & 3 & 2 & 1 & 0 \\ 
-3 & -2 & -4 & -6 & -5 & -4 & -2 & -1 \\ 
2 & 1 & 3 & 4 & 3 & 2 & 1 & 1 
\end{array}\right],
\quad
\gamma\sprime:=
\left[\begin{array}{cccccccc} 
1 & 1 & 2 & 2 & 1 & 1 & 1 & 0 \\ 
-2 & -1 & -2 & -4 & -3 & -2 & -2 & -1 \\ 
3 & 2 & 4 & 6 & 5 & 3 & 2 & 1 \\ 
-2 & -2 & -4 & -5 & -4 & -3 & -2 & -1 \\ 
1 & 1 & 2 & 3 & 3 & 3 & 2 & 1 \\ 
-1 & -1 & -1 & -2 & -2 & -2 & -1 & 0 \\ 
-1 & 0 & -1 & -1 & -1 & -1 & -1 & -1 \\ 
2 & 1 & 2 & 3 & 2 & 2 & 1 & 1 
\end{array}\right] 
\]
}
\caption{Elements of $\OG(E)$ of order $8$}\label{table:order8}
\end{table}
\begin{table}
{
\footnotesize
\[
\begin{array}{c|ccc|c} 
\textrm{no.} & a_1 & a_2 & a_3 & A_{Q} (a) \\
\hline
1 & 568026 & 40203 & 614 & 8 \\ 
2 & 568092 & 40191 & 612 & 16 \\ 
3 & 568290 & 40155 & 606 & 24 \\ 
4 & 568356 & 40143 & 604 & 16 \\ 
5 & 568422 & 40131 & 602 & 16 \\ 
&\dots &&&\\
100 & 580104 & 38007 & 248 & 11240 \\ 
101 & 580170 & 37995 & 246 & 12984 \\ 
102 & 580236 & 37983 & 244 & 14840 \\ 
103 & 580302 & 37971 & 242 & 16712 \\ 
104 & 580368 & 37959 & 240 & 18800 \\ 
105 & 580434 & 37947 & 238 & 20808 \\ 
106 & 580500 & 37935 & 236 & 23184 \\ 
107 & 580566 & 37923 & 234 & 25304 \\ 
108 & 580632 & 37911 & 232 & 27416 \\ 
109 & 580698 & 37899 & 230 & 29720 \\ 
110 & 580764 & 37887 & 228 & 32472 \\ 
&\dots &&&\\
155 & 583734 & 37347 & 138 & 40 \\ 
156 & 583800 & 37335 & 136 & 24 \\ 
157 & 583866 & 37323 & 134 & 16 \\ 
158 & 583932 & 37311 & 132 & 8 \\
 & & &  &  \\ 
 & & &  &  \\ 
 & & &  &  \\ 
\hline
&&&\llap{total} & 1305600
\end{array}
%
%
%
\phantom{aaaaaa}
%
\begin{array}{c|ccc|c} 
\textrm{no.} & a_1 & a_2 & a_3 & A_{Q\sprime} (a) \\
\hline
1 & 568092 & 40191 & 612 & 8 \\ 
2 & 568224 & 40167 & 608 & 24 \\ 
3 & 568290 & 40155 & 606 & 8 \\ 
4 & 568488 & 40119 & 600 & 8 \\ 
5 & 568686 & 40083 & 594 & 16 \\ 
&\dots &&&\\
100 & 580104 & 38007 & 248 & 10688 \\ 
101 & 580170 & 37995 & 246 & 12344 \\ 
102 & 580236 & 37983 & 244 & 14656 \\ 
103 & 580302 & 37971 & 242 & 16064 \\ 
104 & 580368 & 37959 & 240 & 19240 \\ 
105 & 580434 & 37947 & 238 & 20104 \\ 
106 & 580500 & 37935 & 236 & 22984 \\ 
107 & 580566 & 37923 & 234 & 25128 \\ 
108 & 580632 & 37911 & 232 & 28064 \\ 
109 & 580698 & 37899 & 230 & 29128 \\ 
110 & 580764 & 37887 & 228 & 32304 \\ 
&\dots &&&\\
155 & 583734 & 37347 & 138 & 32 \\ 
156 & 583800 & 37335 & 136 & 48 \\ 
157 & 583866 & 37323 & 134 & 24 \\ 
158 & 583932 & 37311 & 132 & 24 \\ 
159 & 583998 & 37299 & 130 & 16 \\ 
160 & 584064 & 37287 & 128 & 8 \\ 
161 & 584196 & 37263 & 124 & 8 \\
\hline
&&&\llap{total} & 1305600
\end{array}
\]
}
\vskip .5cm
\caption{Distributions of intersection patterns of $Q$ and $Q\sprime$}\label{table:CycALs}
\end{table}

Let $\gamma$ be an element of $\OG(E)$ represented by 
the matrix in Table~\ref{table:order8}
with respect to the basis $e_1, \dots, e_8$ of $E$.
Then $\gamma$ is of order $8$ and belongs to the stabilizer subgroup $G^{\typeI}$ 
of $(V_0, W^{\typeI})\in \DDD$.  
Let $v=(v_1, \dots, v_8)\in V_0^8$ be 
\[
\left(
\begin{array}{c|c|c|c|c|c|c|c}
2 2 1 0& 0 1 2 0& 0 2 0 1& 1 0 0 1& 0 2 0 1& 0 2 2 2& 0 1 2 2& 0 1 2 2 
\end{array}
\right),
\]
where each component $v_i$ is written with respect to the basis  of $V_0$ in Table~\ref{table:basisvs}.
Then the subspace $B(\gamma, v)$ of $V_0^8$ satisfies $\ptwo$-condition,
and we obtain a Quebbemann lattice $Q:=Q(\Delta^{\typeI}, B(\gamma, v))$
with an automorphism $\tilde{\gamma}_Q$ of order $8$,
where $\Delta^{\typeI}\in \DDD^8$ is given in the previous subsection.
By the method of $\Gamma$-rigidifying basis,
we see that $\OG(Q)=\{\pm 1\}\times \gen{\tilde{\gamma}_Q}$.
The action of $\OG(Q)$ decomposes  $\Min(Q)$  into $2611200/16=163200$ orbits.
The  distribution  of intersection patterns is given in Table~\ref{table:CycALs} (left).
\par 
Let $\gamma\sprime$ be an element of $\OG(E)$ given in Table~\ref{table:order8},
which is of order $8$ and belongs to the stabilizer subgroup $G^{\typeII}$ 
of $(V_0, W^{\typeII})\in \DDD$.
Let $v\sprime\in V_0^8$ be 
\[
\left(
\begin{array}{c|c|c|c|c|c|c|c}
2220&0102&2120&2220&2202&1202&2220&2112
\end{array}
\right).
\]
Then $B(\gamma\sprime, v\sprime)$  satisfies $\ptwo$-condition,
and we obtain a Quebbemann lattice  
$Q\sprime:=Q(\Delta^{\typeII}, B(\gamma\sprime, v\sprime))$.
We see that $\OG(Q\sprime)=\{\pm 1\}\times \gen{\tilde{\gamma}_{Q\sprime}}$,
and its  action decomposes  $\Min(Q\sprime)$  into $163200$ orbits.
The  distribution  of intersection patterns is given in Table~\ref{table:CycALs} (right).
\bibliographystyle{plain}
\bibliography{myrefs}
%

\end{document}